\documentclass[11pt]{article}
\usepackage{amsmath,amsthm,amssymb,amscd,fancybox,ifthen}
\usepackage[all]{xy}

\newcommand{\rst}{\restrictedto}

\DeclareMathOperator\eo{eo}
\DeclareMathOperator\ooee{oe}

\newcommand\sig{\operatorname{sig}}

\newcommand\grad{\operatorname{grad}}

\newcommand\spnc{\operatorname{Spin}_{\bbC}}

\newcommand\Tr{\operatorname{tr}}

\newcommand\cT{\mathcal T}

\newcommand\cU{\mathcal U}

\newcommand\cR{\mathcal R}
\newcommand\cP{\mathcal P}

\newcommand\bc{\boldsymbol c}

\newcommand\tcP{\widetilde{\cP}}
\newcommand\tcR{\widetilde{\cR}}
\newcommand\tcS{\widetilde{\cS}}
\newcommand\tL{\widetilde{L}}

\newcommand\cE{\mathcal E}

\newcommand\bX{\overline{X}}
\newcommand\bY{\overline{Y}}

\newcommand\dbar{\bar{\pa}}
\newcommand\dbnc{\dbar\rho\rfloor}

\newcommand\teta{\widetilde{\eta}}

\newcommand\tX{\widetilde{X}}

\newcommand\hX{\widehat{X}}
\newcommand\hQ{\widehat{Q}}
\newcommand\heth{\widehat{\eth}}

\newcommand{\Spn}{S\mspace{-10mu}/ }
\newcommand\hSpn{\widehat{\Spn}}

\newcommand\Ker{\operatorname{ker}}

\newcommand\cC{\mathcal{C}}
\newcommand\cL{\mathcal{L}}
\newcommand\cS{\mathcal{S}}

\newcommand\tsigma{\tilde\sigma}

\newcommand\bbC{\mathbb C}

\newcommand\bbN{\mathbb N}
\newcommand\bbP{\mathbb P}
\newcommand\bbR{\mathbb R}

\newcommand\pa{\partial}

\newcommand\restrictedto{\upharpoonright}

\newcommand\ins[1]{{#1}^{\circ}}

\newcommand\CI{{\mathcal C}^{\infty}}

\newcommand\CIc{{\mathcal C}^{\infty}_{\text{c}}}

\newcommand\CmIc{{\mathcal C}^{-\infty}_{\text{c}}}

\newcommand\Id{\operatorname{Id}}

\DeclareMathOperator{\range}{range}
\DeclareMathOperator{\odd}{o}
\DeclareMathOperator{\even}{e}
\DeclareMathOperator{\Ind}{Ind}

\DeclareMathOperator{\Rind}{R-Ind}

\newtheorem{theorem}{Theorem}
\newtheorem{proposition}{Proposition}
\newtheorem{corollary}{Corollary}

\theoremstyle{definition}
\newtheorem{definition}{Definition}

\theoremstyle{remark}
\newtheorem{remark}{Remark}

\begin{document}

\title{Cobordism, Relative Indices and\\
Stein Fillings} 

\author{Charles L. Epstein\footnote{  E-mail: cle@math.upenn.edu\newline
This material is based upon work supported by the National
Science Foundation under Grant No. 0603973, and the Francis J. Carey term
chair. Any opinions, findings, and conclusions or recommendations expressed in
this material are those of the author and do not necessarily reflect the
views of the National Science Foundation.} \\
Department of Mathematics\\ University of Pennsylvania}

\date{Draft: May 8, 2007}

\maketitle
\centerline{\it Dedicated with gratitude and admiration to Gennadi Henkin}
\centerline{\it on the occasion of his 65th birthday.}

\begin{abstract}  In this paper we build on the framework developed 
  in~\cite{Epstein4,Epstein3, Epstein44} to obtain a more complete
  understanding of the gluing properties for indices of boundary value problems
  for the $\spnc$-Dirac operator with sub-elliptic boundary conditions. We
  extend our analytic results for sub-elliptic boundary value problems for the
  $\spnc$-Dirac operator, and gluing results for the indices of these boundary
  problems to $\spnc$-manifolds with several pseudoconvex (pseudoconcave)
  boundary components. These results are applied to study Stein fillability for
  compact, 3-dimensional, contact manifolds.
\end{abstract}

\section*{Introduction}
In several earlier papers we analyzed Fredholm boundary value problems for the
$\spnc$-Dirac operator defined by modifying the $\dbar$-Neumann boundary
condition. To apply this analysis we require a compact, $2n$-dimensional,
$\spnc$-manifold, $X,$ with contact boundary, $Y.$ The $\spnc$-structure must
be defined in a neighborhood, $U,$ of $bX$ by an almost complex structure, $J,$
see~\cite{GKK}.  The contact structure on $Y$ is assumed to be compatible with
the almost complex structure in a sense explained below. In our earlier work we
assume that the boundary $Y$ is a connected manifold. In this paper we extend
the analytic results for sub-elliptic boundary value problems to manifolds with
several boundary components, some pseudoconvex and some pseudoconcave. These
results are then applied the prove various extensions, to the sub-elliptic
case, of Bojarski's gluing formul{\ae} for indices of Dirac operators. Finally
these results are applied, along with the classical excision theorem for
indices of Gromov and Lawson, to study the set of embeddable structures on a
3d-CR manifold.

The almost complex structure, $J,$ defines a splitting of
$TX\otimes\bbC\restrictedto_{U}$ into complementary subbundles
\begin{equation}
TX\otimes\bbC\restrictedto_{U}=T^{1,0}X\oplus T^{0,1}X,
\end{equation}
the dual splitting of $T^*X\otimes\bbC$ is denoted by $\Lambda^{1,0}X\oplus
\Lambda^{0,1}X.$ Though these bundles are only defined in the subset of $X$
where $J$ is defined; to avoid introducing excessive notation, we denote them
by $\Lambda^{1,0}X,$ etc.  This splitting leads to the definition of the
$\dbar$-operator:
\begin{equation}
\dbar f=df\restrictedto_{T^{0,1}X};
\end{equation}
$\dbar f$ is a section of $\Lambda^{0,1}X.$

For each $0\leq p,q\leq n,$ we let $\Lambda^{p,q}$ denote the bundle of forms
of type $(p,q)$ defined by the almost complex structure.  If $\Spn$ denotes the
bundle of complex spinors over $X,$ then over $U$ we have the identification:
\begin{equation}
\Spn\restrictedto_{U}=\bigoplus\limits_{q=0}^n\Lambda^{0,q}X\restrictedto_{U}.
\label{eqn1}
\end{equation}
For each $q,$ the $\dbar$-operator extends to define a map
\begin{equation}
\dbar:\CI(U;\Lambda^{p,q}X)\longrightarrow\CI(U;\Lambda^{p,q+1}X).
\end{equation}
We select an Hermitian metric $g$ on $T^{1,0}X,$ this defines a formal adjoint
$\dbar^*.$ Using the identification in~\eqref{eqn1}, the $\spnc$-Dirac
operator, $\eth$ can be expressed,  over $U,$  as
\begin{equation}
\eth=\dbar+\dbar^*+\cE,
\end{equation}
where $\cE:\Spn\to\Spn$ is a bundle endomorphism.

In this paper, we generally regard manifolds with boundary as closed, so that,
for example, $\rho\in\CI(X)$ means that $\rho$ is smooth up to, and including
the boundary. The notation $\bX$ is used to denote the oriented manifold $X$
with its orientation reversed.

Let $\rho\in\CI(X)$ be a defining function for $bX:$ $X=\{x\in
X:\: \rho(x)<0\},$ $d\rho$ is non-vanishing along $bX.$ The Hermitian
metric on $T^{1,0}X$ defines the interior product operation
\begin{equation}
\dbar\rho\rfloor :\Lambda^{p,q}X\longrightarrow \Lambda^{p,q-1}X.
\end{equation}
The classical $\dbar$-Neumann condition for  sections
$\sigma^{p,q}\in\CI(X;\Lambda^{p,q}X),$ is the requirement
that
\begin{equation}
\dbar\rho\rfloor\sigma^{p,q}\restrictedto_{bX}=0.
\end{equation}

The boundary of $X$ is assumed to be a contact manifold. The contact structure
is compatible with $J$ in that the hyperplane field $H$ on $Y$ is the
null-space of the real 1-form
\begin{equation}
\theta=i\dbar\rho\restrictedto_{TY}.
\label{eqn8}
\end{equation}
In order for our analytic results to apply, the boundary of $X$ must satisfy
one of several
convexity properties, which are described by the signature of the
Levi-form,
\begin{equation}
\cL_y(X,Y)=\frac{1}{2}\left[d\theta_y(X,JY)+d\theta_y(Y,JX)\right],\text{ for }X\in H_y.
\end{equation}
A boundary point $y$ is strictly pseudoconvex if $\cL_y$ is positive definite,
and strictly pseudoconcave if $\cL_y$ is negative definite. Let $Y_j$ be a
connected component of $Y;$ if $\cL_y>0$ ($\cL_y<0$) for all $y\in Y_j$ then we
say that $Y_j$ is strictly pseudoconvex (pseudoconcave). In our earlier papers
we showed how to modify the $\dbar$-Neumann condition to obtain a sub-elliptic
boundary condition provided that each boundary component of $X$ is either
strictly pseudoconvex or strictly pseudoconcave. In fact our approach applies
so long as $\cL_y$ is non-degenerate at every boundary point. The
modifications to the $\dbar$-Neumann condition, needed to define a
sub-elliptic boundary value problem, depend on the signature of $\cL.$ In this
paper we again focus on boundaries that are either pseudoconvex or pseudoconcave.

In the integrable, strictly pseudoconvex case the reason that the
$\dbar$-Neumann condition itself does not define a Fredholm operator for $\eth$
is that $\dbar$ has an infinite dimensional null-space in degree 0, i.e. the
holomorphic functions. The reason is simply that
$\dbar\rho\rfloor\sigma^{0,0}\restrictedto_{bX}=0$ is always satisfied for a
$(0,0)$-form. To correct this we need to change the boundary condition in
degree 0. In the classical case there is an orthogonal projector, $\cS$ defined
on $\CI(bX),$ whose range consists of the boundary values of holomorphic
functions; it is called ``the'' Szeg\H o projector. We distinguish this case,
by calling this a \emph{classical} Szeg\H o projector.

 The boundary condition is modified in degree zero by requiring
\begin{equation}
\cS(\sigma^{0,0}\restrictedto_{bX})=0.
\end{equation}
To get a formally self adjoint operator, the boundary
condition in degree 1 must also be modified by requiring
\begin{equation}
(\Id-\cS)[\dbnc\sigma^{0,1}\restrictedto_{bX}]=0.
\end{equation}
These conditions, along with the $\dbar$-Neumann condition in degrees greater
than 1, define a projector, $\cR_+,$ acting of sections of
$\Spn\restrictedto_{bX}.$ The modified $\dbar$-Neumann condition for $\eth$ on a
strictly pseudoconvex manifold is requirement that
\begin{equation}
\cR_+[\sigma\restrictedto_{bX}]=0.
\label{eqn2}
\end{equation}
The pair $(\eth,\cR)$ denotes the operator defined by $\eth$ acting on a domain
defined by the condition in~\eqref{eqn2}. In our earlier papers we showed that
this operator is essentially self adjoint, and it graph closure is a Fredholm
operator. The spin-bundles and operators split into even and odd parts. The
index of the even part $(\eth^{\even},\cR^{\even}),$ computes the renormalized
holomorphic Euler characteristic of $X:$
\begin{equation}
\Ind(\eth^{\even},\cR^{\even})=\sum\limits_{q=1}^{n}(-1)^q\dim H^{0,q}(X).
\end{equation}

The analytic results are generalized to the non-integrable case by introducing the
notion of a generalized Szeg\H o projector. This idea appears
in~\cite{EpsteinMelrose} and is closely related to that introduced in the
appendix to~\cite{BoutetdeMonvel-Guillemin1}. Briefly, the \emph{contact}
structure on $Y$ defines an algebra of pseudodifferential operators,
$\Psi_H^*(Y),$ called the Heisenberg algebra,
see~\cite{Beals-Greiner1,Taylor3}. The classical Szeg\H o projector, $\cS,$ is
an element of $\Psi_H^0(Y).$ The principal Heisenberg-symbol of $\cS$ is
defined by the complex structure induced on the fibers of $H.$ Generally, if
$(Y,H)$ is a contact manifold, then an almost complex structure, $J,$ on the
fibers of $H,$ is positive if the induced Levi-form is positive definite. This
data defines a function, $s_{J},$ on $T^*Y,$ which is, in turn, the principal
symbol of an operator $\cS\in\Psi_H^0(Y).$
\begin{definition}
An operator $\cS\in\Psi_H^0(Y)$ is a generalized Szeg\H o projector if
\begin{enumerate}
\item $\cS^2=\cS$ and $\cS^*=\cS.$
\item There is a positive almost complex structure $J$ on $H$ so that the
  principal symbol of $\cS$ satisfies:
\begin{equation}
\sigma^H_0(\cS)=s_J.
\end{equation}
\end{enumerate}
\end{definition}

Classical Szeg\H o projectors, defined in the integrable case, are generalized
Szeg\H o projectors, but more importantly, generalized Szeg\H o projectors
exist on any contact manifold with positive almost complex structures.  A
fundamental fact about generalized Szeg\H o projectors is that if $\cS_1$ and
$\cS_2$ are two generalized Szeg\H o projectors on $(Y,H),$ then the restriction
\begin{equation}
\cS_1:\range\cS_2\longrightarrow \range\cS_1
\end{equation}
is a Fredholm operator, see~\cite{EpsteinMelrose}. We denote its index by
$\Rind(\cS_2,\cS_1).$ A generalized Szeg\H o projector is \emph{not} determined by
its full symbol, indeed, amongst pairs $(\cS_1,\cS_2),$ such that $\cS_1-\cS_2$ are smoothing
operators, the relative index $\Rind(\cS_2,\cS_1)$ assumes all integral values.

Using generalized Szeg\H o projectors, the modified pseudoconvex
$\dbar$-condition can be defined on any strictly pseudoconvex $\spnc$-manifold,
$X,$ satisfying the conditions described above. Let $(Y,H)$ be the boundary
of $X,$ which we suppose is strictly pseudoconvex, and let
$\cS\in\Psi_H^0(Y),$ be a generalized Szeg\H o projector. Using the
identification in~\eqref{eqn1}, the modified pseudoconvex $\dbar$-Neumann
condition defined by $\cS$ is given by
\begin{equation}
\begin{split}
&\cS[\sigma^{0,0}\restrictedto_{bX}]=0\\
&(\Id-\cS)[\dbnc\sigma^{0,1}\restrictedto_{bX}]=0\\
&[\dbnc\sigma^{0,q}]\restrictedto_{bX}=0\text{ for }q\geq 2.
\end{split}
\label{eqn3}
\end{equation}
As before these conditions are define by a projector, $\cR_+$ acting on
$\CI(Y;\Spn\restrictedto_{bX}).$ 

\begin{definition} Let $X$ be a manifold with boundary, $E,F$ two smooth vector
  bundles over $X,$ and $P:\CI(X;E)\to \CI(X;F)$ a first order differential
  operator. If $B$ is a pseudodifferential operator acting on sections of
  $E\restrictedto_{bX},$ then $(P,B)$ denotes the differential operator with
  domain $s\in\CI(X;E)$ satisfying $B[s\restrictedto_{bX}]=0.$
\end{definition}

In~\cite{Epstein44} it is shown that if $X$ is strictly pseudoconvex, then
$(\eth,\cR_+)$ is an essentially self adjoint operator and its graph closure is
a Fredholm operator. If $(\eth^{\eo},\cR_+^{\eo})$ are the even and odd parts,
then it is also shown that the adjoints satisfy
\begin{equation}
(\eth^{\eo},\cR_+^{\eo})^*=\overline{(\eth^{\ooee},\cR_+^{\ooee})}.
\end{equation}
Below we show that if $X$ is strictly pseudoconcave, then the same results hold
with $\cR_+$ replaced by $\Id-\cR_+.$

In our earlier papers extensive usage is made of gluing constructions, and
various formul{\ae} are proved relating the indices of sub-elliptic boundary
value problems on the pieces to the index of $\eth^{\even}$ on a boundary-less
glued space. In the first part of this paper we extend these results to more
general situations allowing multiple boundary components, and a glued space
with boundary components. These results are extensions of results of Bojarski
in the elliptic case to the sub-elliptic case. As part of this analysis, we
consider the structure of the Calderon projector on a $\spnc$-manifold with
several boundary components.

In the second part of the paper we apply these results to study the problem of
embeddability (or Stein fillability) for CR-structures on compact
3-manifolds. Let $X_+$ be a strictly pseudoconvex surface with boundary the
CR-manifold $(Y, T^{0,1}_bY).$ We suppose that $(Y,T^{0,1}_b)$ is also the
boundary of a strictly pseudoconcave manifold $X_-,$ which contains a positive,
compact holomorphic curve, $Z.$ Our main result is
\begin{theorem}\label{thm1} Let $(Y,T^{0,1}_bY)$ satisfy the conditions above, and let
  $\cS_0$ denote the classical Szeg\H o projector defined by the CR-structure
  on $Y.$ If
\begin{equation}
H^2_c(X_-;\Theta)=0\text{ and }\deg NZ\geq 2g-1,
\label{eqn18}
\end{equation}
where $g$ is the genus of $Z,$ then there is a constant $M,$ such that for a
sufficiently small embeddable deformations of the CR-structure, with Szeg\H o
projector $\cS_1,$ the relative index satisfies:
\begin{equation}
|\Rind(\cS_0,\cS_1)|\leq M.
\end{equation}
\end{theorem}
As a corollary of this result we conclude that the set of small embeddable
deformations of $(Y,T^{0,1}_bY)$ is closed in the $\CI$-topology. This theorem
is a considerable generalization of the seminal result of Lempert treating
domains in $\bbC^2,$ see~\cite{Lempert2}. It represents the culmination of the
line of research begun in~\cite{EpsteinHenkin2,Epstein5}. It is proved by
combining the index formula from~\cite{Epstein44} with the Gromov-Lawson
excision theorem,~\cite{GromovLawson}, and results of Stipsicz on
the topology of Stein fillings of circle bundles over Riemann surfaces,
see~\cite{stipsicz}.

\section{$\spnc$-boundaries}

Let $X$ be a $2n$-dimensional $\spnc$-manifold with compatible metric $g.$ The
$\spnc$-structure on $X$ defines a bundle, $\Spn,$ of complex spinors, which is
a Clifford module for the complexified Clifford bundle of $T^*X.$ If $dV$ is volume
form, then the Clifford action of $i^n\bc(dV)$ splits $\Spn$ into two subbundles
\begin{equation}
\Spn=\Spn^{\even}\oplus \Spn^{\odd}.
\end{equation}
The Clifford action of $\eta_x\in T^*_xX,$ a non-vanishing 1-form at $x,$  defines
isomorphisms:
\begin{equation}
\bc(\eta_x):\Spn^{\eo}_x\longrightarrow \Spn^{\ooee}_x.
\end{equation}

If $X$ is a manifold with boundary,  then the $\spnc$-structure on $X$
induces a $\spnc$-structure on $bX.$ The spin-bundle of $bX,$ $\Spn_{bX},$ is
canonically isomorphic to $\Spn^{\even}\restrictedto_{bX}.$ Let $t$ be a
defining function for $bX,$ such that $t<0$ on $X,$ $\|dt\|_g=1,$ and $\grad_g t$
is orthogonal to $TbX\subset TX\restrictedto_{bX}.$ Under this
identification, the Clifford action of $\eta\in T^*_{x}bX,$ on $\Spn_{bX}$ is given by
\begin{equation}
\bc_{bX}(\eta)\cdot s=\bc_X(-dt)\bc_X(\teta)\cdot s.
\label{spnb}
\end{equation}
Here $\teta$ is the extension of $\eta$ to $T_xX$ by zero on the orthogonal
complement of $T_xbX\subset T_xX.$

\begin{definition}
Let $(Y,g_Y)$ be an odd-dimensional $\spnc$-manifold, such that there is an
even dimensional $\spnc$-manifold, $(X,g_X)$ with oriented boundary $Y.$
Suppose that $g_X\rst{TY}=g_Y,$ and the $\spnc$-structure on $Y$ satisfies
\begin{equation}
\Spn_Y\simeq \Spn_X^{\even}\restrictedto_{bX},
\end{equation}
and, under this identification, the Clifford action of $T^*Y$ on $\Spn_Y$
satisfies~\eqref{spnb}, with $bX=Y.$ In this case we say that $(Y,\Spn_Y)$ is the
\emph{$\spnc$-boundary} of $(X,\Spn_X).$
\end{definition}
In this connection we often consider the boundary with its orientation
reversed, $\bY.$ Identifying $\Spn_Y$ with $\Spn^{\odd}_X$ and
defining the Clifford action by
\begin{equation}
\bc_{bX}(\eta)\cdot s=\bc_X(dt)\bc_X(\teta)\cdot s,
\label{spnb-}
\end{equation}
defines a $\spnc$-structure on $Y$ inducing the opposite orientation. 

We briefly review the construction of an ``invertible double,'' given
in~\cite{BBW}. Let $X$ be a $\spnc$-manifold with boundary $Y,$ connected or
not. The tubular neighborhood theorem implies that there is a neighborhood,
$U,$ of the boundary that is diffeomorphic to $Y\times [-1,0].$ Using this
identification, we define the double of $X$ to be the oriented manifold
\begin{equation} 
\hX= X\amalg_{bX} \bX.
\end{equation}
Here $\bX$ denote $X$ with the orientation reversed. The boundary,
$Y\times\{0\},$ is now a separating hypersurface in $\hX$ with neighborhood
$V\simeq Y\times (-1,1).$ A function on $\hX$ is smooth near to $Y$ if smooth
with respect to this product structure. The tubular neighborhood theorem
implicitly defines a function, $t$ in $V,$ taking values in $[-1,1].$ We denote
the component of $\hX\setminus Y\times\{0\}$ where $t<0$ by $X_+$ and the other
component by $X_-.$

We can deform the metric in the collar so that
\begin{equation}
g\restrictedto_{V}=dt^2+g_Y,
\end{equation}
where $g_Y$ is a metric on $Y.$ In this case $dt$ is an outward pointing, unit co-vector and
$\bc(-dt),$ Clifford multiplication by $-dt$ defines  unitary isomorphisms of the
spin-bundles
\begin{equation}
\bc(-dt):\Spn^{\even}\restrictedto_{V}\to\Spn^{\odd}\restrictedto_{V}.
\end{equation}
The spin-bundle $\hSpn\to\hX$ is defined by using this identification to glue
$\Spn^{\eo}\restrictedto_{V}$ to $\Spn^{\ooee}\restrictedto_{V}.$ In~\cite{BBW}
it is shown that the Dirac operator extends to act on sections of $\hSpn,$ with
trivial kernel and co-kernel. Hence $(\hX,\hSpn)$ is called an invertible
double.

 Suppose that $X_+$ is a $\spnc$-manifold with boundary $Y$ and $X_-$ is a
$\spnc$-manifold with boundary $\bY.$ Suppose further that the
$\spnc$-structures on $bX_{\pm}$ are (after a change of orientation on one)
isotopic. After attaching cylinders, diffeomorphic to $Y\times [0,1],$ to
$X_+,X_-,$ an obvious modification of the invertible double construction
from~\cite{BBW} provides a $\spnc$-manifold diffeomorphic to $X_+\amalg_Y X_-.$
If the $\spnc$-structures near to $bX_{\pm}$ are defined by almost complex
structures, then Lemma 8 of~\cite{Epstein44} shows that this remains true in
the added cylinders. The original manifolds $X_+, X_-$ are $\spnc$-isomorphic
to open subsets of the glued space. In the sequel it should be understood that
$X_+\amalg_Y X_-$ refers to the $\spnc$-manifold obtained by such an
augmentation and gluing process. In circumstances where there are several
manifolds, we use the notation $\eth_X$ to indicate the $\spnc$-Dirac operator
on the manifold $X.$

In this paper we usually assume that $X$ is a $\spnc$-manifold
with boundary and the $\spnc$-structure is defined in a neighborhood of $bX$ by
an almost complex structure.  In this circumstance the almost
complex structure defines a hyperplane field on $bX,$ as the $\Ker\theta,$
defined in~\eqref{eqn8}. We usually assume that this hyperplane field is a
contact structure, and that, as boundaries of $X,$ each boundary component is
either strictly pseudoconvex or strictly pseudoconcave. For the sake of brevity,
in the sequel we describe this circumstance by the phrase: ``the
$\spnc$-structure is defined in a neighborhood of $bX$ by an almost complex
structure, making the components of $bX$ either strictly pseudoconvex or
strictly pseudoconcave.''

\section{The Calderon projector}\label{sec2}
In our earlier papers we showed that the index of a sub-elliptic boundary value
problem for the $\spnc$-Dirac operator, $(\eth^{\even},\cR^{\even})$ equals the
relative index between the projector defining the boundary condition,
$\cR^{\even}$ and the Calderon projector, $\cP^{\even}:$ 
\begin{equation}
\Ind(\eth^{\even},\cR^{\even})=\Rind(\cP^{\even},\cR^{\even}).
\end{equation}
This relative index can in turn be computed as a difference of traces, which
provides an incisive analytic tool for studying the properties of these indices
under deformation and gluing. In this paper we consider manifolds with several
boundary components. Many of the analytic results in~\cite{BBW} and our earlier
papers are essentially microlocal and so the number of boundary components is
of no import. To analyze the behavior of the indices of boundary value problems
under gluing it is useful to have more detailed information about the Calderon
projector.

Suppose for example that $X$ is a manifold with two boundary components $Y_0,
Y_1.$ The Calderon projector acts on sections of 
\begin{equation}
\Spn\restrictedto_{bX}\simeq
\Spn\restrictedto_{Y_0}\oplus\Spn\restrictedto_{Y_1}
\simeq \Spn_{0}\oplus\Spn_{1}.
\end{equation}
We can use this splitting to write the Calderon projector in block form
\begin{equation}
\cP=\left(\begin{matrix} \cP_{00} & \cP_{01}\\
\cP_{10} & \cP_{11}\end{matrix}\right).
\label{eqn22}
\end{equation}
The principal symbols of the diagonal terms are projectors, the off-diagonal
terms are smoothing operators.  It is of interest to when when this
projector can be deformed, through projectors, to a diagonal matrix. A simple
analytic sufficient condition is that $\cP_{00}$ and $\cP_{11}$ are
projectors. 
\begin{proposition}\label{prop1} Suppose that $\cP$ is a projector with block
  form as in~\eqref{eqn22}. If $\cP_{00}^2=\cP_{00}$ and $\cP_{11}^2=\cP_{11},$
  then
\begin{equation}
\cP_{t,s}=\left(\begin{matrix} \cP_{00} & t\cP_{01}\\
s\cP_{10} & \cP_{11}\end{matrix}\right).
\label{eqn222}
\end{equation}
are projectors for all $t,s\in\bbC.$
\end{proposition}
\begin{proof} The fact that $\cP^2=\cP,$ coupled with the equations
  $\cP_{00}^2=\cP_{00}$ and $\cP_{11}^2=\cP_{11},$ imply that
\begin{equation}
\begin{split}
&\cP_{10}\cP_{01}=0=\cP_{01}\cP_{10}\\
&\cP_{00}\cP_{10}+\cP_{10}\cP_{11}=\cP_{10}\\
&\cP_{01}\cP_{00}+\cP_{11}\cP_{01}=\cP_{01}.
\end{split}
\label{eqn34}
\end{equation}
From~\eqref{eqn34} we easily deduce that
\begin{equation}
\cP_{t,s}=\left(\begin{matrix} \cP_{00} &t \cP_{10}\\
s\cP_{01} & \cP_{11}\end{matrix}\right)
\end{equation}
is a projector,  for $s,t\in\bbC.$
\end{proof}

\begin{remark}
Evidently, $\cP_{t,t}$ defines a homotopy through projectors, from $\cP=\cP_{1,1}$ to a
block diagonal matrix, $\cP_{0,0}.$
\end{remark}

In this section we consider the structure of the Calderon projector for a Dirac
operator on a manifold with several boundary components. It is important to
understand that any fundamental solution for $\eth$ leads to the construction
of a Calderon projector. If $X$ is a $\spnc$-manifold with boundary, we can
assume that $\ins{X}$ is a relatively compact open subset of larger $\spnc$-manifold,
$\tX.$ We let $r$ denote a defining function for $bX,$ such that
$dr\restrictedto_{bX}$ is orthogonal to $T^*bX$ and $\|dr\|=1.$ Let $\eth$
denote the Dirac operator on $\tX.$ Assume that $Q$ is a fundamental solution
defined on $\tX;$ thus if $s\in\CIc(\tX;\Spn),$ then
\begin{equation}
Q\eth s=\eth Qs=s.
\end{equation}
It is clear that $Q\CIc(\tX;\Spn)\subset\CI(\tX;\Spn)$ and therefore, by
duality, we can extend $Q$ to act on $\CmIc(\tX;\Spn).$ 

Let $u\in\CI(X;\Spn),$ satisfy $\eth u=0,$ and let $U$ denote the extension
of $u$ by zero to all of $\tX.$ We see that $\eth
U=\bc(dr)u\restrictedto_{bX}\otimes\delta(r),$ and therefore:
\begin{equation}
Q[\bc(dr)u\restrictedto_{bX}\otimes\delta(r)]\restrictedto_X=u.
\end{equation}
More generally, if $f$ is a section of $\Spn\restrictedto_{bX},$ then
\begin{equation}
F=Q[\bc(dr)f\otimes\delta(r)]
\end{equation}
belongs to $\Ker\eth$ on $\tX\setminus bX.$ The analysis in Chapter 12
of~\cite{BBW} (for example) shows that $F$ has well defined limits as we
approach $bX,$ from either side, which we denote by $\cP_{\pm}f.$ We use $+$ to
denote the limit from $X$ and $-,$ the limit from $\tX\setminus X.$ The
discussion above shows that $\cP_+$ acts as the identity on the boundary values
of harmonic spinors defined in $X.$ 

 Let $Q_0$ and $Q_1$ be two fundamental solutions defined in a neighborhood of
$X,$ and $\cP_{0+}, \cP_{1+}$ the Calderon projectors they define.  Because two
fundamental solutions differ by a smoothing operator, it follows that
$(1-t)Q_0+tQ_1$ is also a fundamental solution for any $t\in [0,1].$ This
implies that any pair of Calderon projectors are strongly isotopic:
\begin{proposition} Let $\cP_{0+}, \cP_{1+}$ be Calderon projectors defined by
  fundamental solutions for $\eth_{X},$ then there is a smooth path,
  $\{\cP_{t+}:\: t\in [0,1]\},$ in the space
  of pseudodifferential projections joining $\cP_{0+}$ to $\cP_{1+}.$
\end{proposition}
This result allows us to be a bit sloppy about which Calderon projector we are
using.

In what follows we are usually more specific as to the origin of the
fundamental solution. Indeed, $\tX$ is usually taken to be a compact, closed
manifold on which $\eth$ is invertible.  The range of $\cP_-$ consists of the
boundary values of harmonic spinors on $\tX\setminus X,$ and we have the jump
formula:
\begin{equation}
\cP_++\cP_-=\Id.
\label{jmpform}
\end{equation}

Denote the Dirac operator on the invertible double, $\hX,$ by $\heth.$ Since
$\heth$ is invertible, there is a fundamental solution, $\hQ,$ defined on
$\hX,$ which is a classical pseudodifferential operator of order $-1.$ The
Calderon projector, $\cP_+,$ for $\eth$ on $X_+\simeq X$ is a
pseudodifferential operator defined on $bX$ whose range consists of the
boundary values of harmonic spinors on $X_+,$ that is, solutions to
\begin{equation}
\eth \sigma=0
\end{equation}
in $\CI(X_+;\Spn).$ In~\cite{BBW} the fundamental solution $\hQ$ is used to
construct a Calderon projector. As noted in~\eqref{jmpform}  its nullspace consists of
boundary values of harmonic spinors on $X_-.$ 

If $D\subset X_+$ is any domain with smooth boundary, then the
Boo\ss-Bavnbeck-Wojciechowski construction applies, {\it mutatis mutandis,} to
construct a Calderon projector, $\cP_{D_+},$  defined on $bD.$  The range of
$\cP_{D_+}$ consists of the boundary values of harmonic spinors defined on $D,$ and
its nullspace consists of boundary values of harmonic spinors defined on the
complement $\hX\setminus D.$ If we denote this complement by $D_-,$ then this
statement is simply the identity:
\begin{equation}
\cP_{D_+}+\cP_{D_-}=\Id.
\label{eqn31}
\end{equation}
The proof of the following result is now quite simple:
\begin{proposition}\label{prop2} Let $X_{01}$ be a $\spnc$-manifold with boundary consisting
  of two components, $Y_0, Y_1.$ Suppose that $\bY_0$ is the boundary of a
  $\spnc$-manifold with boundary $X_0.$ There is a Calderon projector, $\cP,$ for
  $\eth_{X_{01}},$ so that with respect to the splitting in~\eqref{eqn22}, the
  diagonal terms satisfy:
\begin{equation}
\cP_{00}^2=\cP_{00}\text{ and }\cP_{11}^2=\cP_{11}.
\end{equation}
\end{proposition}
\begin{proof} We let $X_1$ denote the $\spnc$-manifold obtained by gluing $X_0$
  to $X_{01}.$ To do this gluing it may be necessary to glue collars onto $X_0$
  and $X_{01},$ in which to flatten the $\spnc$-structure. The important point
  is that $X_{01}$ is an open subset of $X_1.$ We can then double $X_1$
  across its remaining boundary, to obtain the invertible double $\hX_1.$ Let
  $\hQ$ denote the fundamental solution for $\heth$ on $\hX_1.$ Using this
  fundamental solution we construct Calderon projectors for $X_{01},$ $X_1$ and
  $\ins{(X_1\setminus X_{01})},$ which we denote by $\cP, \cP_1,$ and
  $\cP_0.$ In light of the construction of a Calderon projector as a limit,
  and~\eqref{eqn31}, a moments thought shows that the block decomposition of
  $\cP$ takes the form:
\begin{equation}
\cP=\left(\begin{matrix} \Id-\cP_0 & P_{10}\\
P_{01} & \cP_1\end{matrix}\right).
\end{equation}
As $\cP_0$ and $\cP_1$ are projectors, the assertion of the proposition follows.
\end{proof}

Propositions~\ref{prop1} and ~\ref{prop2} imply:
\begin{corollary} Under the hypotheses of Proposition~\ref{prop2}, the Calderon
  projector can be deformed through projectors to
\begin{equation}
\cP_d=\left(\begin{matrix} \Id-\cP_0 & 0\\
0 & \cP_1\end{matrix}\right).
\end{equation}
\end{corollary}
\begin{remark} Note that if $\bY_0$ is a $\spnc$-boundary, then $\bY_1$ is
  the $\spnc$-boundary of $\bX_1.$
\end{remark}

These results have a natural generalization when $X_0$ has many boundary
components. Suppose that $bX_0$ has several components, $Y_1,\dots, Y_N.$ We
group these boundary components into disjoint (non-empty) subsets
\begin{equation}
\begin{split}
&Y^j=\amalg_{l=m_j}^{m_{j+1}-1}Y_{l},\quad j=1,\dots, J,\text{ where }\\
&1=m_1<m_2<\cdots<m_J<m_{J+1}=N+1,
\end{split}
\end{equation}
with the property that each collection $\bY^j$ is the boundary of a
$\spnc$-manifold, $X_j.$ Gluing along these collections of boundary components
we obtain a $\spnc$-manifold $X^1=X_0\amalg X_2\amalg\cdots\amalg X_J,$ with
boundary equal to $Y^1.$ Let $\hX^1$ denote the invertible double of $X^1,$ and
$\hQ$ its fundamental solution. There is a Calderon projector for $X_0$ that
can be deformed to a block diagonal matrix, with one block for each collection
of boundary components $Y^j.$

For each $2\leq j\leq J,$ let $\cP^j$ denote the Calderon projector, defined by
$\hQ,$ for the manifold $X_j,$ and $\cP^1$ the Calderon projector for $X^1.$
For each $2\leq k\leq J,$ let $\tcP^k$ denote the Calderon projector, defined
by $\hQ,$ for the manifold $\tX_k=X_0\amalg X_k\amalg\cdots\amalg X_J.$ With these
preliminaries we can state the following theorem.
\begin{theorem} With $X_0, Y^1,\dots, Y^J$ as above, the Calderon projector,
  $\cP^0,$ for $X_0,$ can be deformed through projectors to the block diagonal
  matrix:
\begin{equation}
\cP^0_0=\left(\begin{matrix} \cP^1&0 && \cdots& 0\\
0&\Id-\cP^2& 0&\cdots &\vdots\\
\vdots&&\ddots&& 0\\
0&\cdots&&&\Id-\cP^J
\end{matrix}
\right)
\end{equation}
\end{theorem}
\begin{proof} We split $\Spn\restrictedto_{bX}$ into
\begin{equation}
\Spn\restrictedto_{bX}=\Spn\restrictedto_{Y^1\amalg\cdots\amalg Y^{J-1}}\oplus
\Spn\restrictedto_{Y^J}.
\end{equation}
In the notation introduced before the theorem  the projector, $\cP^0$ then takes the form
\begin{equation}
\cP^0=\left(\begin{matrix} \tcP^{J} & A_J\\
B_J & \Id-\cP^J\end{matrix}\right).
\end{equation}
Here $A_J, B_J$ are smoothing operators. As $\cP^0, \tcP^{J}$ and $\cP^J$ are all
projectors,  Proposition~\ref{prop1} shows
that 
\begin{equation}
\cP^0_{1t}=
\cP^0=\left(\begin{matrix} \tcP^{J} & tA_J\\
tB_J & \Id-\cP^J\end{matrix}\right)
\end{equation}
is a one parameter family of projectors, and we can therefore deform to
\begin{equation}
\cP^0_{10}=\left(\begin{matrix} \tcP^{J} & 0\\
0 & \Id-\cP^J\end{matrix}\right).
\end{equation}
For $3 \leq k\leq J,$ we see that
\begin{equation}
\Spn\restrictedto_{b\tX^k}\simeq\Spn\restrictedto_{Y^1\amalg\cdots\amalg
  Y^{k-2}}\oplus \Spn\restrictedto_{Y^{k-1}},
\end{equation} and, with respect to this splitting, the projector $\tcP^k,$ is of the form
\begin{equation}
\tcP^{k}=\left(\begin{matrix} \tcP^{k-1} & A_{k-1}\\
B_{k-1} & \Id-\cP^{k-1}\end{matrix}\right).
\end{equation}
Repeating this argument recursively, along with the fact that $\cP^1=\tcP^2,$
leads to a homotopy through projectors (with constant block diagonal) from
$\cP^0$ to $\cP^0_0.$
\end{proof}

\section{Analysis on manifold with several boundary components}
In this section we study the index of the $\spnc$-Dirac operator on a manifold
with several boundary components, some pseudoconvex and some pseudoconcave. For
example, let $X_{01},$ have two boundary components, $Y_0, Y_1.$ We suppose
that the $\spnc$-structure on $X_{01}$ is induced by almost complex structures
in neighborhoods of its boundary components. We also assume that $Y_1$ is
pseudoconvex and $Y_0$ is pseudoconcave, with respect to the corresponding
almost complex structures.  The boundary components, $Y_0, Y_1$ are contact
manifolds. We let $\cS_0,\cS_1$ be generalized Szeg\H o projectors defined on
$(Y_0,H_0),$ $(Y_1,H_1),$ respectively. Along with the almost complex
structures, these define projectors, $\cR_{0+}, \cR_{1+}$ acting on sections of
the spin-bundle restricted to the boundary. Let $\eth_{X_{01}}$ denote the
$\spnc$-Dirac operator on $X_{01}.$ We let $(\eth_{X_{01}},
[(\Id-\cR_{0+}),\cR_{1+}])$ denote the $\spnc$-Dirac operator acting on smooth
spinors $\sigma,$ which satisfy:
\begin{equation}
(\Id-\cR_{0+})[\sigma\restrictedto_{Y_0}]=0\text{ and
  }\cR_{1+}[\sigma\restrictedto_{Y_1}]=0.
\end{equation}

In~\cite{Epstein3,Epstein44} we established the analytic properties of these boundary
value problems by studying the comparison operator:
\begin{equation}
\cT=\cR\cP+(\Id-\cR)(\Id-\cP),
\end{equation}
here $\cR$ is either the pseudoconvex or pseudoconcave modification of the
$\dbar$-Neumann condition and $\cP$ is a Calderon projector for $\eth.$  The
analytic results  follow from the existence of a parametrix, $\cU,$ for $\cT$
satisfying 
\begin{equation}
\begin{split}
\cT\cU&=\Id-K_1\\
\cU\cT&=\Id-K_2,
\end{split}
\label{parmeqn}
\end{equation}
where $K_1, K_2$ are smoothing operators on $bX.$ The operator $\cU$ belongs to
the extended Heisenberg calculus on $bX$ and its construction is entirely
microlocal. The input from the $\spnc$-structure/$\spnc$-Dirac operator is that
coming from the principal symbol of the Calderon projector.

Thus far, we have only given the complete details of this construction for
$(\eth,\cR_+)$ on a strictly pseudoconvex $\spnc$-manifold. Because these
results rest entirely upon the construction of $\cU,$ they also hold for
$\Id-\cR_+$ on a strictly pseudoconcave manifold: Clearly $\cT$ is symmetric in
$\cR$ and $\Id-\cR.$ When combined with the fact that the construction of $\cU$
uses only the principal symbol of the Calderon projector, and
$\sigma_0(\cP_-)=\sigma_0(\Id-\cP_+),$ we see that it makes no difference
whether we are working on the boundary of a pseudoconvex manifold using the
boundary condition $\cR_+,$ or on a pseudoconcave manifold using the boundary
condition $\Id-\cR_+.$

Let $X$ be a $\spnc$-manifold with boundary, $Y=Y_1\amalg\cdots\amalg Y_N.$
Suppose that an almost complex structure is defined in a neighborhood of the
each boundary component, inducing the given $\spnc$-structure, such that each
boundary component is either strictly pseudoconvex on strictly
pseudoconcave. Let $\cP_+$ be the Calderon projector defined on $X$ by
including $X$ into a compact $\spnc$-manifold, $\tX$ with an invertible Dirac
operator. Let $\cP_{-}$ denote the Calderon projector for $\tX\setminus X;$ it
is important that
\begin{equation}
\cP_++\cP_-=\Id.
\label{eqn54}
\end{equation}
For a Calderon projector defined by embedding $X$ into a compact manifold with
invertible Dirac-operator the following result, which is Proposition 11
in~\cite{Epstein44}, holds:
\begin{proposition}\label{prop11}
Let $X$ be a $\spnc$-manifold with boundary embedded into $\tX$ a compact
$\spnc$-manifold with invertible Dirac operator. Let $t$ be a defining function
for $bX$ such that $t<0$ on $X,$ $\grad_g t$ is orthogonal to $TbX$ and
$\|dt\|=1$ along $bX.$ If $\cP^{\eo}_{\pm}$ are Calderon projectors defined by
the fundamental solution to $\eth$ on $\tX$ then
\begin{equation}
\cP^{\eo}_{\pm}=\bc(\pm dt)\cP^{\ooee}_{\mp}\bc(\pm dt)^{-1}.
\label{pmreln}
\end{equation}
\end{proposition}
\begin{proof} In the statement of Proposition 11
  in~\cite{Epstein44} it is assumed  that $\tX$ is an invertible double for
  $X,$ however this hypothesis is not used in the proof.  All that is needed is
  the assumption that the Dirac operator on $\tX$ is invertible and the
  Calderon projector is constructed using the fundamental solution defined on $\tX.$
\end{proof}

With respect to the splitting
\begin{equation}
\Spn\rst_{bX}=\Spn\rst_{Y_1}\oplus\cdots\oplus \Spn_{Y_N},
\end{equation}
the Calderon projector takes the form:
\begin{equation}
\cP=\left(\begin{matrix} \cP_{11} &\cP_{12}&\cdots&\cP_{1n}\\
\cP_{21} &\cP_{22}&\cdots&\cP_{2n}\\
\vdots&\vdots&&\vdots\\
\cP_{n1}&\cP_{n2}&\cdots&\cP_{nn}\end{matrix}\right).
\end{equation} 
Usually we will make assumptions that imply $\cP_{jj}^2=\cP_{jj}$ for $1\leq
j\leq N,$ but in all cases $\cP_{jj}^2-\cP_{jj}$ and $\cP_{jk}$ for $j\neq k$
are smoothing operators.  For each $j$ we choose a generalized Szeg\H o
projector, $\cS_j\in\Psi^{0}_{H_j}(Y_j).$ Let $\cR_{j+}$ denote the modified
pseudoconvex $\dbar$-Neumann condition defined by $\cS_j.$ For a pseudoconvex
boundary component, $Y_j,$ we let
\begin{equation}
\cT_{j}^{+}=\cR_{j+}\cP_{jj}+(\Id-\cR_{j+})(\Id-\cP_{jj}),
\end{equation}
for a pseudoconcave boundary component, $Y_k,$ we let
\begin{equation}
\cT^{-}_{k}=(\Id-\cR_{k+})\cP_{kk}+\cR_{k+}(\Id-\cP_{kk}).
\end{equation}
Define the function $\epsilon_j=+$ if $Y_j$ is pseudoconvex and $-$ otherwise.

The remarks above easily imply the following result.
\begin{proposition}\label{prop3} For each boundary component, $Y_j$ the operator
  $\cT_{j}^{\epsilon_j}$ is an elliptic element in the extended Heisenberg
  algebra. There is a parametrix $\cU_{j}^{\epsilon_j}$ so that, for smoothing
  operators $K_{j1}, K_{j2},$ we have
\begin{equation}
\cT_{j}^{\epsilon_j}\cU_{j}^{\epsilon_j}=\Id-K_{j1}\text{ and }
\cU_{j}^{\epsilon_j}\cT_{j}^{\epsilon_j}=\Id-K_{j2}.
\end{equation}
\end{proposition} 

Now we order the boundary components so that  $Y_1,\dots, Y_L$ are strictly
pseudoconcave and $Y_{L+1},\dots,Y_N$ are strictly pseudoconvex. Set
\begin{equation}
\cR=\left(\begin{matrix}\Id-\cR_{1+}& 0 &\cdots&0&\cdots&&0\\
\vdots&\ddots&&\vdots&&&\vdots\\
0&\cdots&\Id-\cR_{L+}&0&\cdots&&0\\
0&\cdots&0&\cR_{(L+1)+}&0&\hdots&\vdots\\
\vdots&&&&\ddots&&0\\
0&\cdots&0&0&\cdots&0&\cR_{N+}\end{matrix}\right),
\label{eqn61}
\end{equation}
and let 
\begin{equation}
\cT=\cR\cP+(\Id-\cR)(\Id-\cP).
\end{equation}
The following relationship between the chiral parts $\cR^{\even}$ and
$\cR^{\odd}$ is a consequence of the formal self adjointness of $\cR;$ it is
proved in~\cite{Epstein44}.
\begin{proposition} The chiral parts satisfy:
\begin{equation}
\cR^{\even}=\bc(dt)(\Id-\cR^{\odd})\bc(dt)^{-1}
\label{eqn63}
\end{equation}
\end{proposition}

If we define $\cU$ to be the diagonal matrix with diagonal
\begin{equation}
\cU=\text{diag}(\cU_{1}^{-},\dots,\cU_{L}^{-},
\cU_{L+1}^{+},\dots,\cU_{N}^{+}),
\end{equation}
then Proposition~\ref{prop3}, and the fact that the off-diagonal elements in
$\cP$ are smoothing operators implies the following basic result:
\begin{theorem}\label{thm3} The operator $\cU$ is a parametrix for $\cT.$
\end{theorem}
\begin{proof} Let $\cP_d$ denote the diagonal of $\cP,$ and
  $\cP_{od}=\cP-\cP_d.$ If we let $\cT_d=\cR\cP_d+(\Id-\cR)(\Id-\cP_d),$ then
  Proposition~\ref{prop3} implies that
\begin{equation}
\cU\cT_{d}-\Id\text{ and }\cT_{d}\cU-\Id
\end{equation}
are smoothing operators. As $\cT-\cT_{d}=(2\cR-\Id)\cP_{od}$ is a smoothing
operator it follows immediately that $K_1$ and $K_2$ in
\begin{equation}
\cT\cU=\Id-K_1\text{ and }\cU\cT=\Id-K_2
\label{eqn66}
\end{equation}
are also smoothing operators.
\end{proof}
In the case that the diagonal  of $\cP$ is a projector, this argument gives a
stronger result.
\begin{corollary}\label{cor2} Suppose that $\cP_d^2=\cP_d;$  define
\begin{equation}
\cP_t=\cP-t\cP_{od}\text{ and }
\cT_t=\cR\cP_t+(\Id-\cR)(\Id-\cP_t).
\end{equation}
For each $t,$ $\cP_t$ is a projector, and $\cU$ is a parametrix for $\cT_t,$
with
\begin{equation}
\cT_t\cU=\Id-K_{1t}\text{ and }\cU\cT_t=\Id-K_{2t}.
\end{equation}
The operators $\{(K_{1t}, K_{2t}):\: t\in [0,1]\}$ are a smooth family of
smoothing operators.
\end{corollary}
\begin{remark} Note that
\begin{equation}
\cP_0=\cP\text{ and }\cP_1=\cP_d.
\end{equation}
\end{remark}

\section{The relative index formula}
Recall that if $\sigma$ and its distributional derivative, $\eth\sigma,$ both
belong to $L^2(X;\Spn),$ then $\sigma$ has a well defined restriction to $bX$
as an element of the Sobolev space $H^{-\frac 12}(bX;\Spn\restrictedto_{bX}).$
Theorem~\ref{thm3} combined with the arguments in~\cite{Epstein44} prove the
following result:
\begin{theorem}\label{thm4} Let $X$ be a $\spnc$-manifold with boundary, such
  that the $\spnc$-structure is defined in a neighborhood of $bX$ by an almost
  complex structure, making each boundary component of $X$ either strictly
  pseudoconvex or strictly pseudoconcave. If we define the domain for $\eth$ to be
\begin{equation}
\{\sigma\in L^2(X;\Spn):\: \eth \sigma\in L^2(X;\Spn),\,
\cR[\sigma\restrictedto_{bX}]=0\},
\end{equation}
where $\cR$ is defined as in~\eqref{eqn61}, then $\eth$ is a Fredholm
  operator. There is a constant $C$ so that if $\sigma$ satisfies these
  conditions, then
\begin{equation}
\|\sigma\|_{H^{\frac 12}(X)}\leq
  C[\|\eth\sigma\|_{L^2(X)}+\|\sigma\|_{L^2(X)}].
\end{equation}
The chiral restrictions $\eth^{\eo}$ are Fredholm and their $L^2$-adjoints
satisfy
\begin{equation}
[(\eth^{\eo},\cR^{\eo})]^*=\overline{(\eth^{\ooee},\cR^{\ooee})}.
\end{equation}
\end{theorem}
\begin{remark}
Indeed, there are also higher norm estimates: For each $s\geq 0,$ there is a
constant $C_s$ so that if $\sigma\in L^2,$ $\eth \sigma\in H^s,$ and
$\cR[\sigma\restrictedto_{bX}]=0,$ then $\sigma\in H^{s+\frac 12},$ and
\begin{equation}
\|\sigma\|_{H^{s+\frac 12}}\leq C_s[\|\eth\sigma\|_{H^s}+\|\sigma\|_{L^2}].
\end{equation}
These estimates imply that the null-space of $\eth$ is contained in
$\CI(X;\Spn).$
\end{remark}

As in our earlier papers, the indices of $(\eth^{\eo},\cR^{\eo})$ can be
computed as the relative indices on the boundary between $\cP^{\eo}$ and $\cR^{\eo}.$
Theorem~\ref{thm3} shows that $\cP^{\eo}$ and $\cR^{\eo}$ are a tame Fredholm pair, and
therefore the relative index can be computed as the index of:
\begin{equation}
\Rind(\cP^{\eo},\cR^{\eo})=\Ind[\cR^{\eo}:
\cP^{\eo}\CI(bX;\Spn\restrictedto_{bX})\longrightarrow
\cR^{\eo}\CI(bX;\Spn\restrictedto_{bX})].
\label{eqn68}
\end{equation}
\begin{theorem}\label{thm5} Let $X$ be a compact $\spnc$-manifold as in
  Theorem~\ref{thm4}. Suppose that $\cP$ is a Calderon projector for
  $\eth_{X},$ which satisfies
\begin{equation}
\cP^{\even *}=\bc(dt)(\Id-\cP^{\odd})\bc(dt)^{-1}.
\label{eqn75}
\end{equation}
If $\cR$ is a projector acting on sections of $\Spn\rst{bX}$ as
in~\eqref{eqn61}, then
\begin{equation}
\Ind(\eth^{\eo},\cR^{\eo})=\Rind(\cP^{\eo},\cR^{\eo}).
\end{equation}
\end{theorem}
\begin{remark} If the Calderon projector is defined by embedding $X$ into
  $\tX,$ a closed compact $\spnc$-manifold, with invertible Dirac operator, then
  the relation~\eqref{eqn75} follows from  Proposition~\ref{prop11}
  and~\eqref{eqn54}.
\end{remark}

\begin{proof} We give the proof for the even case, the odd case is essentially
  identical. The null-space of $(\eth^{\even},\cR^{\even})$ consists of smooth
  sections $\sigma$ of $\Spn^{\even}$ satisfying:
\begin{equation}
\eth^{\even}\sigma=0\text{ and }\cR^{\even}[\sigma\restrictedto_{bX}]=0.
\end{equation}
It is clear that $\cP^{\even}[\sigma\rst_{bX}]=\sigma\rst_{bX},$ and
therefore $\sigma\rst_{bX}$ belongs to the null-space of $\cR^{\even}$ acting
on the range of $\cP^{\even}.$ On the other hand, if
$s\in\range\cP^{\even}$ and $\cR^{\even} s=0,$ then there is a unique
harmonic spinor $\sigma,$ with $\sigma\rst_{bX}=s.$ This shows that the
null-space of $(\eth^{\even},\cR^{\even})$ is isomorphic to the null-space of
the restriction in~\eqref{eqn68}.

The co-kernel of $\cR^{\even}\cP^{\even}$ is isomorphic to the null-space of
\begin{equation}
\cP^{\even *}:\range\cR^{\even}\longrightarrow\range\cP^{\even *}.
\end{equation}
Equation~\eqref{eqn63} implies that the range of $\cR^{\even}$ is $\bc(dt)$
applied to the null-space of $\cR^{\odd};$ this, along with~\eqref{eqn75},
shows that the co-kernel of $\cR^{\even}\cP^{\even}$ is isomorphic to the
intersection of the null-space of $\cR^{\odd}$ with the range of $\cP^{\odd}.$
By the first part of the argument, this intersection is isomorphic to
$\Ker(\eth^{\odd},\cR^{\odd}).$ Applying the last statement of
Theorem~\ref{thm4}, we complete the proof of the theorem.
\end{proof}

Using general properties of tame Fredholm pairs it follows that the relative
index can be computed as a difference of traces.
\begin{corollary}\label{cor3} Suppose that the parametrix $\cU$ for $\cT$
  satisfies~\eqref{eqn66}, then
\begin{equation}
\Ind(\eth^{\eo},\cR^{\eo})=\Rind(\cP^{\eo},\cR^{\eo})=
\Tr(\cP^{\eo}K_2^{\eo}\cP^{\eo})-\Tr(\cR^{\eo} K_1^{\eo}\cR^{\eo})
\label{eqn79}
\end{equation}
\end{corollary}
\begin{proof} Because $\cP^{\eo}$ and $\cR^{\eo}$ are tame Fredholm
  pairs, this is an immediate consequence of Theorem 15 in~\cite{Epstein44}.
\end{proof}

As in~\cite{Epstein44} the relative index formula has a useful corollary:
\begin{corollary}\label{cor4} Let $X$ be a compact $\spnc$-manifold with boundary as in
  Theorem~\ref{thm4} and $\cR$ a modified $\dbar$-Neumann boundary condition as
  in~\eqref{eqn61}. If $\cP$ is a Calderon projector for $\eth_{X},$ then 
\begin{equation}
\Rind(\cP^{\eo},\cR^{\eo})=-\Rind((\Id-\cP^{\eo}),(\Id-\cR^{\eo})).
\label{eqn80}
\end{equation}
\end{corollary}
\begin{proof}
It follows from Proposition~\ref{prop2} and Corollary~\ref{cor3} that the
relative indices in~\eqref{eqn80} do not depend on the choice of Calderon
projector, and therefore we can assume that $\cP$ is defined using the
invertible double construction. As it relies only on very general properties of
the Calderon projector, and the invertible double construction, the argument
used to prove Corollary 5 in~\cite{Epstein44} applies, with minor changes, to
establish~\eqref{eqn80}.
\end{proof} 

In the case that the diagonal of $\cP,$ $\cP_{d},$ is itself a projector,
Corollary~\ref{cor2} shows that, for each $t\in[0,1],$ $(\cP_t,\cR^{\eo}),$
where $\cP_{t}^{\eo}=\cP^{\eo}-t\cP^{\eo}_{od},$ is a tame Fredholm pair. The
index of these pairs can also be computed by evaluating a trace:
\begin{equation}
\Rind(\cP_{t}^{\eo},\cR^{\eo})=
\Tr(\cP^{\eo}_{t}K_{2t}^{\eo}\cP^{\eo}_{t})-\Tr(\cR^{\eo}
K_{1t}^{\eo}\cR^{\eo})
\label{eqn81}
\end{equation}
The operators on the right hand side of~\eqref{eqn81} are smoothing operators,
depending smoothly on $t,$ hence the traces depend smoothly on $t$ as well. As
the difference is an integer it must be constant. This proves the following
result.
\begin{theorem}\label{thm6} If $X,\cP,\cR$ satisfy the hypotheses of
  Theorem~\ref{thm5}, and the diagonal of the Calderon projector is itself a
  projector, then
\begin{equation}\Ind(\eth^{\eo},\cR^{\eo})=
  \Rind(\cP_d^{\eo},\cR^{\eo}).
\end{equation}
\end{theorem}

This result is our basic tool for studying the gluing properties of the indices
of sub-elliptic boundary value problems for $\eth.$

\section{Gluing formul{\ae} for the index of $\eth$}

We now consider the behavior of the index of $\eth$ with modified
$\dbar$-Neumann conditions under gluing operations.  This approach was
implicitly used in~\cite{Epstein4}, though we did not directly address the
analytic properties of boundary value problems on manifolds with several
ends. Under this rubric there is a huge multiplicity of possible situations
that one might consider, in
this section we focus on a $\spnc$-manifold, $X_{01},$ with two boundary
components, $Y_0, Y_1.$ As usual, we assume that the $\spnc$-structure, in a
neighborhood of $bX_{01},$ is induced by an almost structure, and that $Y_0,
Y_1$ are contact manifolds, with $Y_0$ strictly pseudoconcave and $Y_1$
strictly pseudoconvex. 

Let $\cS_0, \cS_1$ be generalized Szeg\H o projectors defined on $Y_0, Y_1$
respectively and $\cR_0,\cR_1,$ the pseudoconvex, modified $\dbar$-Neumann
boundary conditions they define. As it is the case of principal interest in
applications to complex analysis, we often assume that $\bY_0$ is also the
pseudoconvex boundary of a compact $\spnc$-manifold $X_0.$ We let $X_1\simeq
X_0\amalg_{Y_0} X_{01},$ denote the $\spnc$-manifold obtained by gluing $X_0$
to $X_{01}$ along $Y_0.$ The operators $(\eth_{X_0},\cR_0), (\eth_{X_1},\cR_1)$
are Fredholm, as is $(\eth_{X_{01}},[\Id-\cR_0,\cR_1]).$ Our basic result is a
gluing formula for $\Ind(\eth_{X_1},\cR_1).$
\begin{theorem}\label{thm7} Let $X_0, X_{01}$ and $X_1$ be as above, with
  $\cR_0, \cR_1$  modified pseudoconvex $\dbar$-Neumann conditions. The indices
  satisfy the following relation:
\begin{equation}
\Ind(\eth^{\eo}_{X_1},\cR^{\eo}_1)=
\Ind(\eth^{\eo}_{X_0},\cR_0)+\Ind(\eth^{\eo}_{X_{01}},[(\Id-\cR^{\eo}_0),\cR^{\eo}_1]).
\label{eqn82}
\end{equation}
\end{theorem}
\begin{proof} To prove this formula, we  express the various indices in terms
  of relative indices on the boundaries. Let $\hX_1$ denote the invertible
  double of $X_1,$ and $\hQ_1$ the fundamental solution for $\eth_{\hX_1}.$ Let
  $\cP_0,\cP_1$ be Calderon projectors, for $\eth_{X_0},\eth_{X_1},$
  respectively, defined by $\hQ_1.$ Finally let $\cP_{01}$ be the Calderon
  projector for $\eth_{X_{01}}$ defined by $\hQ_1.$ The discussion in
  Section~\ref{sec2} shows that
\begin{equation}
\cP_{01}=\left(\begin{matrix} \Id-\cP_0 & P_{10}\\
P_{01} & \cP_1\end{matrix}\right)
\end{equation}
and therefore the diagonal of $\cP_{01}$ is itself a
projector. Theorem~\ref{thm5} shows that
\begin{multline}
\Ind(\eth^{\eo}_{X_0},\cR_0)+\Ind(\eth^{\eo}_{X_{01}},[(\Id-\cR^{\eo}_0),\cR^{\eo}_1])=\\
\Rind(\cP^{\eo},\cR_0)+\Rind(\cP^{\eo}_{01},[(\Id-\cR^{\eo}_0),\cR^{\eo}_1]).
\label{eqn83}
\end{multline}
Theorem~\ref{thm6} applies to show that the second term on the right hand side
of~\eqref{eqn83} can be replaced by
\begin{equation}
\begin{split}
\Rind(\cP^{\eo}_{01},[(\Id-\cR^{\eo}_0),\cR^{\eo}_1])=
&\Rind([(\Id-\cP^{\eo}_0),\cP^{\eo}_1],[(\Id-\cR^{\eo}_0),\cR^{\eo}_1])\\
=&\Rind((\Id-\cP^{\eo}_0),(\Id-\cR^{\eo}_0))+\Rind(\cP^{\eo}_1,\cR^{\eo}_1).
\end{split}
\label{eqn84}
\end{equation}
Finally we apply Corollary 5 from~\cite{Epstein44} to replace
$\Rind((\Id-\cP^{\eo}_0),(\Id-\cR^{\eo}_0))$ with
$-\Rind(\cP^{\eo}_0,\cR^{\eo}_0).$ Once again applying Theorem~\ref{thm5} we
obtain
\begin{equation}
\Ind(\eth^{\eo}_{X_0},\cR_0)+\Ind(\eth^{\eo}_{X_{01}},[(\Id-\cR^{\eo}_0),\cR^{\eo}_1])=
\Ind(\eth^{\eo}_{X_1},\cR^{\eo}_1).
\label{eqn85}
\end{equation}
as desired.
\end{proof}

As a special case we consider $X_{01}=Y_0\times [0,1].$ In this case the
formula can be rewritten as:
\begin{equation}
\Ind(\eth^{\eo}_{X_{01}},[(\Id-\cR^{\eo}_0),\cR^{\eo}_1])=
\Ind(\eth^{\eo}_{X_1},\cR^{\eo}_1)-\Ind(\eth^{\eo}_{X_0},\cR^{\eo}_0).
\label{eqn87}
\end{equation}
Since $X_1$ is homotopic, as a $\spnc$-manifold to $X_0,$ we can consider
$\cR_1$ as defining a boundary condition on $X_0.$ The index of
$(\eth^{\eo}_{X_1},\cR^{\eo}_1)$ does not change as we smoothly deform $X_1$ to
$X_0$, and we can therefore apply the Agranovich-Dynin formula, Theorem 8
from~\cite{Epstein44}, to prove:
\begin{corollary} If $Y$ is a strictly pseudoconvex, contact manifold, bounding
  a $\spnc$-manifold, and $\cS_0,\cS_1$ are generalized Szeg\H o projectors
  defined on $Y,$ then
\begin{equation}
\Rind(\cS_0,\cS_1)=\Ind(\eth^{\even}_{Y\times
  [0,1]},[(\Id-\cR^{\even}_0),\cR^{\even}_1]).
\label{eqn88}
\end{equation}
\end{corollary}
\begin{remark} This result is strongly suggested by the analysis
  in~\cite{Epstein44}, but does not follow directly from it. It is unclear whether
  the result remains true if $Y$ is not the boundary of $\spnc$-manifold.
\end{remark}

Applying Theorem~\ref{thm7} twice we easily obtain a cocycle formula for these
indices.
\begin{corollary}\label{cor5} Suppose that $X_{01}, X_{12}$ are $\spnc$-manifolds with
  boundaries $Y_0\amalg Y_1,$ $Y_1\amalg Y_2,$ respectively. Assume that
  $\bY_0$ is also the pseudoconvex boundary of a compact $\spnc$-manifold. Let
  $\cS_0,\cS_1,\cS_2,$ denote generalized Szeg\H o projectors defined on $Y_0,
  Y_1, Y_2,$ and $\cR_0,\cR_1,\cR_2$ the modified pseudoconvex $\dbar$-Neumann
  boundary conditions they define. The following cocycle relation holds:
  \begin{multline}
\Ind(\eth^{\eo}_{X_{02}},[(\Id-\cR^{\eo}_0),\cR^{\eo}_2])=\\
\Ind(\eth^{\eo}_{X_{01}},[(\Id-\cR^{\eo}_0),\cR^{\eo}_1])+
\Ind(\eth^{\eo}_{X_{12}},[(\Id-\cR^{\eo}_1),\cR^{\eo}_2]).
\end{multline}
\end{corollary}

\begin{remark} As suggested to the author by Laszlo Lempert, one might try to extend
  the notion of the relative index between pairs of generalized Szeg\H o
  projectors defined on one contact manifold, to a relative index between pairs
  of generalized Szeg\H o projectors defined on pairs of ``almost complex
  $\spnc$-cobordant'' contact manifolds, $(Y_0,H_0), (Y_1, H_1).$ By this we
  mean that there is a $\spnc$-manifold with boundary $X_{01}$ such that
  $bX_{01}=Y_1\amalg \bY_0,$ and the $\spnc$-structure on $X_{01}$ is defined
  in a neighborhood of $bX_{01}$ by an almost complex structure. The almost
  complex structure induces the given contact structures on the boundary
  components, and the boundary components are strictly pseudoconvex, resp.
  pseudoconcave.

Let $\cS_0, \cS_1$ be generalized Szeg\H o projectors
  defined on $(Y_0,H_0),$ $(Y_1,H_1),$
  respectively. Generalizing~\eqref{eqn88}, one might attempt to define
\begin{equation}
``\Rind(\cS_0,\cS_1)\text{{}''}=\Ind(\eth^{\even}_{X_{01}},[(\Id-\cR^{\even}_0),\cR^{\even}_1]).
\label{eqn90}
\end{equation}
Corollary~\ref{cor5} shows that this invariant satisfies the cocycle formula.
The problem with this definition is that it seems unlikely that two different
choices of almost complex $\spnc$-cobordism will give the same value for
$\Rind(\cS_0,\cS_1).$ If $X_{01}'$ is another such cobordism, then this amounts
to knowing whether or not $\Ind(\eth^{\even}_{X_{01}\amalg \bX_{01}'})$
vanishes.

While this definition does not appear to be adequate, it seems
likely that one could modify the definition in~\eqref{eqn90} by subtracting a
topological or geometric invariant of the cobordism, $\cT(X_{01}),$ with the
properties that
\begin{equation}
\begin{split}
\cT(X_{01})+\cT(\bX_{01}')&=\Ind(\eth^{\even}_{X_{01}\amalg \bX_{01}'})\\
\cT(Y\times [0,1])&=0.
\end{split}
\end{equation}
The modified invariant would then agree with $\Rind(\cS_0,\cS_1)$ in the
product case, and would depend only on the pair
$(Y_0,H_0,\cS_0),(Y_1,H_1,\cS_1).$  
\end{remark}

In our earliest work on relative indices between classical Szeg\H o projectors
we had a variety of conditions assuring that $\Rind(\cS_0,\cS_1)$ vanishes,
see~\cite{Epstein}. Following the philosophy of the remark, we have a
considerable generalization of our earlier results.
\begin{theorem}\label{thm8} Let $X$ be a strictly pseudoconvex, complex manifold
  with boundary, on which there is defined an exhaustion function, $\varphi.$
  For each $c\in\bbR$ let
\begin{equation}
X_{c}=\varphi^{-1}((-\infty,c])\text{ and }X^c=X\setminus X_c.
\end{equation}
Suppose that for some $c_0,$ $\varphi$ is strictly plurisubharmonic in
$X^{c_0}.$ For  $c>c_0,$ a regular value of $\varphi,$ let $\cS_0$ be the
classical Szeg\H o projector defined on $bX_c,$ $\cS_1,$ the classical Szeg\H o
projector defined on $bX,$ and $\cR_0,\cR_1,$ the modified pseudoconvex
$\dbar$-Neumann boundary conditions they define. Under these assumptions
\begin{equation}
\Ind(\eth^{\even}_{X^c},[(\Id-\cR^{\even}_0),\cR^{\even}_1])=0.
\label{eqn92}
\end{equation}
\end{theorem}
\begin{proof} The gluing formula~\eqref{eqn82} implies that~\eqref{eqn92} is
  equivalent to the statement that
\begin{equation}
\Ind(\eth_{X}^{\even},\cR^{\even}_1)=
\Ind(\eth_{X^c}^{\even},\cR^{\even}_0).
\label{eqn93}
\end{equation}
Since we are working in the integrable case we can apply equation 77
of~\cite{Epstein4} to conclude that
\begin{equation}
\begin{split}
&\Ind(\eth_{X}^{\even},\cR^{\even}_1)=\sum_{q=1}^n(-1)^q\dim H^{0,q}(X)\\
&\Ind(\eth_{X^c}^{\even},\cR^{\even}_0)=\sum_{q=1}^n(-1)^q\dim H^{0,q}(X_c),
\end{split}
\label{eqn94}
\end{equation}
where $n=\dim_{\bbC} X.$ As there is a strictly plurisubharmonic exhaustion
defined in $X^c,$ $(X,X_c)$ is a Runge pair. Hence, we can apply the classical
results of Andreotti, Grauert and H\"ormander to conclude that
\begin{equation}
H^{0,q}(X)\simeq H^{0,q}(X_c)\text{ for }1\leq q\leq n.
\label{eqn95}
\end{equation}
See~\cite{Hormander8}. The theorem follows immediately from~\eqref{eqn94}
and~\eqref{eqn95}. 
\end{proof}

\section{Sub-elliptic boundary conditions along a separating hypersurface}
Suppose that $X$ is a compact Spin- or $\spnc$-manifold and $Y\hookrightarrow X$ is a
separating hypersurface; let $X\setminus Y= X_0\amalg X_1.$  Let $\cP_0,\cP_1$ be
Calderon projectors defined on $X_0,X_1$ respectively. Bojarski's theorem
expresses the $\Ind(\eth_{X})$ as the relative index:
\begin{equation}
\Ind(\eth^{\even}_{X})=\Rind(\cP^{\even}_{1},(\Id-\cP^{\even}_0)).
\end{equation}
If $P$ is a classical pseudodifferential projector acting on $\Spn\rst{Y},$ so
that $P\cP_0+(\Id-P)(\Id-\cP_0)$ is classically elliptic, then
$(\eth^{\even}_{X_0},P^{\even})$ and $(\eth^{\even}_{X_1},(\Id-P^{\even}))$ are
Fredholm operators. Expressing the indices of these operators as relative
indices, and using the cocycle relation for relative indices, Bojarski's
theorem easily implies that
\begin{equation}
\Ind(\eth^{\even}_{X})=\Ind(\eth^{\even}_{X_0},P^{\even})+
\Ind(\eth^{\even}_{X_1},(\Id-P^{\even})).
\label{bjrthm1}
\end{equation}
In~\cite{Epstein4} we generalized this identity to the sub-elliptic case, but
only under the assumption that the $\spnc$-structure on $X$ is defined by an
integrable, almost complex structure. In this section we use the relative index
formalism developed here and in~\cite{Epstein44} to extend this formula to the general case.
\begin{theorem}\label{thm9} Let $X$ be a $\spnc$-manifold and $Y\hookrightarrow X,$ a
  separating hypersurface; let $X\setminus Y= X_0\amalg X_1.$ Suppose that the
  $\spnc$-structure is defined in a neighborhood of $Y$ by an almost complex
  structure, inducing a contact structure on $Y=Y_1\amalg\cdots\amalg Y_N$ with
  definite Levi-form. We suppose that the components $Y_1,\dots,Y_L$ are
  strictly pseudoconcave, and $Y_{L+1},\dots,Y_N$ are strictly pseudoconvex,
  with respect to $X_0.$ For each boundary component we choose a generalized
  Szeg\H o projector, $\{\cS_i:\: i=1,\dots,N\},$ and let $\cR_0$ be the modified
  $\dbar$-Neumann boundary condition they define as in~\eqref{eqn61}, then
\begin{equation}
\Ind(\eth^{\even}_{X})=\Ind(\eth^{\even}_{X_0},\cR^{\even}_0)+
\Ind(\eth^{\even}_{X_1},(\Id-\cR^{\even}_0)).
\label{bjrthm2}
\end{equation}
\end{theorem}
\begin{proof} Let $\cP_0$ and
  $\cP_1$ denote Calderon projectors defined, using the invertible doubles
   $X_0\amalg \bX_0,$ and $X_1\amalg \bX_1,$ on $bX_0$ and $bX_1,$ respectively
   The indices on the right hand side of~\eqref{bjrthm2} can be computed, using
   Theorem~\ref{thm5}, as relative indices:
\begin{equation}
\begin{split}
\Ind(\eth^{\even}_{X_0},\cR^{\even}_0)&=\Rind(\cP^{\even}_0,\cR^{\even}_0)\\
\Ind(\eth^{\even}_{X_1},(\Id-\cR^{\even}_0))
&=\Rind(\cP^{\even}_1,(\Id-\cR^{\even}_0)).
\end{split}
\end{equation}
Corollary~\ref{cor4}  applies to show that
\begin{equation}
\Rind(\cP^{\even}_1,(\Id-\cR^{\even}_0))=-\Rind((\Id-\cP^{\even}_1),\cR^{\even}_0).
\end{equation}
We are left to show that:
\begin{equation}
\Rind(\cP^{\even}_0,(\Id-\cP^{\even}_1))=
\Rind(\cP^{\even}_0,\cR^{\even}_0)-\Rind((\Id-\cP^{\even}_1),\cR^{\even}_0)
\label{eqn115}
\end{equation}
The result then follows from Bojarski's theorem. The proof of~\eqref{eqn115} is
essentially identical to the proof of Proposition 13 in~\cite{Epstein44}.  The
difference here is that in our earlier paper $X_0$ and $X_1$ are both
pseudoconvex, so we worked with $X_0$ and $\bX_1.$ This is why
$(\Id-\cP^{\even}_1)$ appears in the second term of~\eqref{eqn115}, instead of
$\cP^{\even}_1,$ as in equation (204) of~\cite{Epstein4}. The argument
in~\cite{Epstein44} relies on general properties of the parametrix $\cU$ and
indices of tame Fredholm pairs, which are unconnected to the number, or
convexity properties of the boundary components.  The routine modifications
needed to establish~\eqref{eqn115} are left to the reader.
\end{proof}

\section{The non-separating case}
Not yet considered is the case of a non-separating hypersurface $Y$ in a compact
$\spnc$-manifold. We make our usual assumptions regarding the $\spnc$-structure
on $X:$ the structure is induced, in a neighborhood of $Y$ by an almost complex
structure. The almost complex structure defines a contact structure on $Y,$
with respect to which the Levi-form is definite. The manifold with boundary
$X_{01}=X\setminus Y,$ has two boundary components, $Y_0, Y_1,$ both isomorphic
to $Y.$ For simplicity we limit ourselves to the case that $Y$ is connected,
though the results proved here clearly extend to the case that $Y$ has several
components.

Following our practice above, we label the components so that $Y_1$ is a
strictly pseudoconvex boundary and $Y_0,$ a strictly pseudoconcave boundary.
Let $\cS$ be a generalized Szeg\H o projector defined on $Y,$ and $\cR$ the
modified pseudoconvex $\dbar$-Neumann boundary operator it defines.  The
boundary value problems on $X\setminus Y,$
$(\eth^{\eo}_{X_{01}},[(\Id-\cR^{\eo}),\cR^{\eo}])$ are Fredholm. By analogy to
the previous results we would expect that the index of this operator computes
the index of the closed manifold,
\begin{equation}
\Ind(\eth^{\eo}_{X})=\Ind(\eth^{\eo}_{X_{01}},[(\Id-\cR^{\eo}),\cR^{\eo}]).
\end{equation}

To prove this we use a device suggested by~\cite{DaiZhang}: We attach a collar
$\overline{Y\times [0,1]},$ to the boundary of $X_{01}.$ To do this we first flatten
the $\spnc$-structure in a neighborhood of $bX_{01}.$ This does not change
$\Ind(\eth^{\eo}_{X_{01}},[(\Id-\cR^{\eo}),\cR^{\eo}]),$ and 
\begin{equation}
\Ind(\eth_{X})=\Ind(\eth_{X_{01}\amalg \overline{Y\times [0,1]}}).
\end{equation}
Thus Theorem~\ref{thm9} implies the following formula for the index
of $\eth^{\eo}_{X}:$
\begin{multline}
\Ind(\eth^{\eo}_{X})=
\Ind(\eth^{\eo}_{X_{01}},[(\Id-\cR^{\eo}),\cR^{\eo}])-\\
\Ind(\eth^{\eo}_{Y\times[0,1]},[(\Id-\cR^{\eo}),\cR^{\eo}]).
\end{multline}
  We are therefore
reduced to showing that
\begin{equation}
\Ind(\eth^{\eo}_{Y\times[0,1]},[(\Id-\cR^{\eo}),\cR^{\eo}])=0.
\end{equation}
This can easily be established by a direct calculation.

Let $\theta$ be a one-form defining the contact structure $H$ on $Y$ and $J$ a
complex structure on the fibers of $H$ so that $\cL_J=d\theta(\cdot,J\cdot)$ is
positive definite on $H\times H.$  $T$ denotes the Reeb vector field:
$\theta(T)=1, i_Td\theta=0.$ We use $\cL_J$ to define the metric on $H$ and
declare $T$ to be orthogonal to $H$ and of unit length. With this data the
$\spnc$-bundle on $Y$ satisfies
\begin{equation}
\Spn_Y\simeq \bigoplus\limits_{q=0}^{n-1} \Lambda_b^{0,q} Y.
\end{equation}
We realize $\Lambda_b^{0,1}Y$ as a subbundle of $T^*Y\otimes\bbC$ by requiring
the restriction to $T^{1,0}_bY\oplus \{\bbC T\}$ to vanish. Let $\rho$ denote a
coordinate on $[0,1].$ We extend the almost complex structure to $Y\times
[0,1]$ be defining $J \pa_\rho=T,$ and the metric, by declaring $\pa_\rho$ to
have unit length, and to be orthogonal to $TY.$ 

The spin-bundle on $Y\times [0,1]$ is isomorphic to $\Spn_{Y\times
[0,1]}=\bigoplus\Lambda^{0,q}(Y\times [0,1]),$ with the obvious splitting into
even and odd forms. Clearly $\Spn_{Y},$ pulled back to $Y\times [0,1],$ is
canonically a subbundle of $\Spn_{Y\times [0,1]}$ under these
identifications. We can write a section of $\Spn_{Y\times [0,1]}$ in the form
\begin{equation}
\sigma=\sigma_t(\rho)+\dbar\rho\wedge\sigma_n(\rho),
\end{equation}
where $\sigma_t(\rho),\sigma_n(\rho)$ are 1-parameter families of sections of
$\Spn_Y,$ that is elements of $\CI([0,1];\CI(\Spn_Y)).$ If $\sigma$ is a
section of $\Spn^{\even}_{Y\times [0,1]},$ then $\sigma_t$ is a 1 parameter
family of even-degree sections of $\Spn_Y,$ and $\sigma_n$ is a 1 parameter
family of odd-degree sections of $\Spn_Y.$ Analogous statements are true for
sections of $\Spn^{\odd}_{Y\times [0,1]}.$ The isomorphism of
$\Spn^{\even}_{Y\times [0,1]}$ with $\Spn_Y$ just takes
$\sigma^{\even}\to\sigma_t+\sigma_n.$ 

Under this identification the operator $\eth^{\even}_{Y\times [0,1]}$ becomes
\begin{equation}
\eth^{\even}_{Y\times [0,1]}\leftrightarrow \pa_{\rho}+B,
\end{equation}
where $B$ is the self-adjoint Dirac-operator on $Y.$ As $\cL_J$ is
positive definite, the end $Y\times \{1\}$ is strictly pseudoconvex and
$Y\times \{1\}$ is strictly pseudoconcave. The boundary condition
$[(\Id-\cR^{\even}),\cR^{\even}]$ becomes:
\begin{equation}
\begin{split}
\cS\sigma_t^{0,0}(1)=0&\quad \sigma_n(1)=0\\
(\Id-\cS)\sigma_t^{0,0}(0)=0&\quad \sigma_t(0)=0.
\end{split}
\end{equation}
The odd-part $(\eth^{\odd}_{Y\times [0,1]},[(\Id-\cR^{\odd}),\cR^{\odd}])$ is
the adjoint of $(\eth^{\even}_{Y\times [0,1]},[(\Id-\cR^{\even}),\cR^{\even}])$
and so, under these identifications, we have
\begin{equation}
\eth^{\odd}_{Y\times [0,1]}\leftrightarrow -\pa_{\rho}+B,
\end{equation}
and
the boundary condition, $[(\Id-\cR^{\odd}),\cR^{\odd}]$ becomes:
\begin{equation}
\begin{split}
(\Id-\cS)\sigma_t^{0,0}(1)=0&\quad \sigma_t(1)=0\\
\cS\sigma_t^{0,0}(0)=0&\quad \sigma_n(0)=0.
\end{split}
\end{equation}
With these preliminaries, it is now easy to see that the kernel and cokernel of
$(\eth^{\even}_{Y\times [0,1]},[(\Id-\cR^{\even}),\cR^{\even}])$ are isomorphic and
therefore: 
\begin{equation}
\Ind(\eth^{\even}_{Y\times [0,1]},[(\Id-\cR^{\even}),\cR^{\even}])=0.
\label{eqn125}
\end{equation}
Suppose that $\sigma(\rho)\in\CI([0,1];\CI(\Spn_Y))$ represents an
element of the null-space of this operator. Clearly
$\tsigma(\rho)=\sigma(1-\rho),$ then belongs to the null-space of
$(\eth^{\odd}_{Y\times [0,1]},[(\Id-\cR^{\odd}),\cR^{\odd}]).$ As this is the
adjoint operator, the assertion of~\eqref{eqn125} follows immediately.

This completes the proof of the following theorem:
\begin{theorem}
Let $X$ be a compact $\spnc$-manifold and $Y\hookrightarrow X,$ a
non-separating hypersurface. Suppose that the $\spnc$-structure is induced, in
a neighborhood of $Y$ by an almost complex structure, with respect to which $Y$
is a contact manifold with a definite Levi-form. Let $X_{01}=X\setminus Y,$ and
$\cS$ be a generalized Szeg\H o projector defined on $Y,$ with $\cR$ the
modified pseudoconvex $\dbar$-Neumann boundary operator it defines. We have
that:
\begin{equation}
\Ind(\eth^{\eo}_{X})=\Ind(\eth^{\eo}_{X_{01}},[(\Id-\cR^{\eo}),\cR^{\eo}]).
\end{equation}
\end{theorem}

\section{Stein fillings for 3-manifolds}
We now show how to use the gluing results for the relative index to prove our
main result, Theorem~\ref{thm1}. For this result we assume that $(Y,H)$ is a
compact 3-dimensional, contact manifold with a strictly pseudoconvex
CR-structure, $T^{0,1}_bY,$ supported by $H,$ that arises as the boundary of a
strictly pseudoconvex complex manifold, $X_+.$ Let $\cS_0$ denote the classical
Szeg\H o projector onto boundary values of holomorphic functions defined on
$X_+.$ In addition we assume that $(Y, T^{0,1}_bY)$ arises as the pseudo{\it
  concave} boundary of a smooth complex manifold with boundary $X_-,$ and that
$X_-$ contains a positive, smooth, compact holomorphic curve, $Z.$ By positive
we mean that there is a strictly plurisubharmonic exhaustion function,
$\varphi$ defined in $X_-\setminus Z,$ so that $bX_-=\varphi^{-1}(0),$ and
$\varphi(x)$ tends to infinity as $x\to Z.$ We extend $\varphi$ smoothly to
$X_+$ so that $X_+=\varphi^{-1}((-\infty,0]).$ For $c\in\bbR,$ we let
\begin{equation}
X_{c}=\varphi^{-1}((-\infty,c]).
\end{equation}

\begin{proof}[Proof of Theorem~\ref{thm1}]
The hypothesis of the theorem includes the requirement that
$H^2_c(X_-;\Theta)=0.$ The basic result of Kiremidjian implies that any
sufficiently small perturbation, $\omega,$ of the CR-structure on $bX_-$ can be
extended to define an integrable deformation, $\Omega,$ of the complex
structure on $X_-.$ If we choose a sufficiently large $c\in\bbR,$ then
$Y_c=\varphi^{-1}(c)$ is the strictly pseudoconcave boundary of small tubular
neighborhood of $Z.$ The manifold $Y_c$ is diffeomorphic to a circle bundle in
the normal bundle to $Z,$ $NZ=T^{1,0}X_-\rst_{Z}/T^{1,0}Z.$ Indeed, it is not
difficult to show that the contact structure on $Y_c$ is isotopic to the
standard $U(1)$-invariant contact structure on the unit circle in $NZ$ defined
by a metric on $NZ$ with positive curvature.

The $\dbar$-operator defined by the deformed complex structure,
$\dbar_{\Omega}$ satisfies:
\begin{equation}
\dbar_{\Omega}=\dbar_0+P_{\Omega},
\end{equation}
where $P_\Omega$ is a first order operator with smooth coefficients bounded in
the $\cC^1$-topology by $C\|\Omega\|_{\cC^k},$ for some $C\in\bbR,k\in\bbN.$
Using the Banach space version of Kiremidjian's theorem proved
in~\cite{EpsteinHenkin3}, it follows that for another $C', k'$ these
coefficients are bounded in the $\cC^1$-topology by $C'\|\omega\|_{\cC^{k'}}.$
If we fix a $c\in\bbR,$ as above, then, provided that $\|\omega\|_{\cC^{k'}}$
is sufficiently small, the exhaustion function $\varphi$  remains strictly
plurisubharmonic, with respect to $\dbar_{\Omega},$ on $X_c\cap X_-.$

Now suppose that the deformed CR-structure on $Y$ is fillable, and so it can
also be realized as the boundary of strictly pseudoconvex complex manifold,
$X_+'.$ We let $X'=X_+'\amalg X_-',$ where $X_-'$ denotes $X_-$ with the
deformed complex structure defined by $\Omega.$.  Let $\cS_1$ denote the Szeg\H
o projector onto the boundary values of holomorphic functions defined on
$X_+',$ $\tcS_0$ the Szeg\H o projector on $bX_c$ with respect to the original
complex structure, and $\tcS_1$ the Szeg\H o projector on $bX_c',$ with respect
to the deformed complex structure. To prove the theorem we show that
\begin{equation}
\Rind(\cS_0,\cS_1)=\Rind(\tcS_0,\tcS_1).
\label{eqn99.0}
\end{equation}
From the hypothesis we know that $\deg NZ\geq 2g-1,$ where $g$ is the genus of
$Z.$ Thus $bX_c$ is covered by the Theorem of Stipsicz: Amongst Stein fillings
of a circle bundle of degree $d$ over a surface with genus $g,$ with the
standard contact structure, if $d\geq 2g-1,$ then the signature and Euler
characteristic are bounded, see~\cite{stipsicz}.  Using the formula
from~\cite{Epstein44}:
\begin{multline}
\Rind(\tcS_0,\tcS_1)=
\dim H^{0,1}(X_c)-\dim H^{0,1}(X_c')+\\
\frac{\sig(X_c)-\sig(X_c')+\chi(X_c)-\chi(X_c')}{4},
\end{multline}
we conclude that $\Rind(\tcS_0,\tcS_1)$ assumes only finitely many values. Note
that the Stipsicz result has no smallness assumption on the size of the
perturbation of the CR-structure.

We let $\cR_0,\cR_1,\tcR_0,\tcR_1$ denote the modified pseudoconvex
$\dbar$-Neumann boundary conditions defined by these Szeg\H o
projectors. Because we have strictly plurisubharmonic exhaustion functions
defined on the collars,
\begin{equation}
X_{c-}=X_c\cap X_-\text{ and }X_{c-}'=X_c'\cap X_-',
\end{equation}
we can apply Theorem~\ref{thm8} to conclude that
\begin{equation}
\begin{split}
\Ind(\eth_{X_c},\tcR_0)&=\Ind(\eth_{X_+},\cR_0)\\
\Ind(\eth_{X_c'},\tcR_1)&=\Ind(\eth_{X_+'},\cR_1).
\end{split}
\label{eqn99}
\end{equation}

We can add a collars to both pairs, $X_c, \bX'_c$ and $X_+,\bX_+',$ to
obtain compact $\spnc$-manifolds, $\hX_c\simeq X_c\amalg \bX'_c,$ $\hX_+\simeq
X_+\amalg \bX_+',$ respectively. Theorem 9 of~\cite{Epstein44} applies to show
that
\begin{equation}
\begin{split}
\Rind(\cS_0,\cS_1)&=\Ind(\eth^{\even}_{\hX_+})-\Ind(\eth_{X_+},\cR_0)+
\Ind(\eth_{X_+'},\cR_1),\\
\Rind(\tcS_0,\tcS_1)&=\Ind(\eth^{\even}_{\hX_c})-\Ind(\eth_{X_c},\tcR_0)+
\Ind(\eth_{X_c'},\tcR_1).
\end{split}
\end{equation}
Combining these formul{\ae} with those in~\eqref{eqn99} we see that
\begin{equation}
\Rind(\cS_0,\cS_1)-\Rind(\tcS_0,\tcS_1)=\Ind(\eth^{\even}_{\hX_+})-
\Ind(\eth^{\even}_{\hX_c}).
\end{equation}
Finally we can deform the $\spnc$-structure on $\hX_c$ to obtain a
$\spnc$-manifold $\hX_c'\simeq X_+\amalg X_{c-}\amalg\bX_{c-}\amalg X_+''.$
Here $X_+''$ is the $\spnc$-manifold, $X_+'$ with a collar attached deforming
the $\spnc$-structure on $bX_+'$ to that defined on $bX_+.$
Clearly this deformation does not change
$\Ind(\eth^{\even}_{\hX_c}),$ moreover $X_+\amalg X_+''\simeq \hX_+.$

The excision theorem of Gromov and Lawson (see Chapter 10 of~\cite{BBW})
applies to show that
\begin{equation}
\begin{split}
\Ind(\eth^{\even}_{\hX_c'})&=\Ind(\eth^{\even}_{X_+\amalg X_+''})+
\Ind(\eth^{\even}_{X_{c-}\amalg\bX_{c-}})\\
&=\Ind(\eth^{\even}_{\hX_+}).
\end{split}
\end{equation}
The second term vanishes because $X_{c-}\amalg\bX_{c-}$ is an invertible
double. This completes the proof of~\eqref{eqn99.0}, and thereby the proof of the
theorem. 
\end{proof}

One might reasonably enquire when the geometric hypotheses in
equation~\eqref{eqn18} hold.  A simple case to consider is that of line bundle
over a Riemann surface, $L\to \Sigma.$  Let $g$ denote the genus of $\Sigma$
and $d=\deg L.$  In~\cite{EpsteinHenkin2} we compute $H^2_c(X_;\Theta),$ where
$X_-$ is a neighborhood of the zero section in $L.$ We use the $S^1$-action to
decompose $H^2_c(X_-;\Theta)$ into Fourier components:
\begin{equation}
H^2_c(X_-;\Theta)\simeq\bigoplus\limits_{k=-1}^{\infty}
H^2_c(X_-;\Theta)_{(k)}.
\end{equation}
With $\kappa$ the canonical bundle of $\Sigma,$ the Fourier components fit into
long exact sequences:
\begin{equation}
[H^2_c(X_-;\Theta)_{(-1)}]'\simeq H^0(\Sigma;\kappa^2\otimes L^{-1}),
\end{equation}
for $k\geq 0:$
\begin{equation}
\begin{split}
H^0(\Sigma;\kappa\otimes L^{-k-2})&\longrightarrow[H^2_c(X_-;\Theta)_{(k)}]'
\longrightarrow\\
 &H^0(\Sigma;\kappa^2\otimes L^{-k-2})\longrightarrow
H^0(\Sigma;L^{k+2})\longrightarrow\cdots
\end{split}
\end{equation}
If $\deg L\geq 3g-3,$ then $\deg\kappa^2\otimes L^{-1}\leq g-1,$ and
generically
\begin{equation}
[H^2_c(X_-;\Theta)_{(-1)}]'\simeq H^0(\Sigma;\kappa^2\otimes L^{-1})=0,
\end{equation}
see~\cite{GriffithsHarris}. The other Fourier components are easily seen to
vanish. This improves upon our earlier result where we proved a similar bound on the
relative index assuming that $d>4g-3.$ This proves the following:
\begin{proposition} Suppose that $L\to \Sigma$ is a  line bundle over a surface,
  with $\deg L$ at least $3g-3,$ where $g$ is the genus of $\Sigma.$ Let $\tL$
  denote the compactification of $L$ obtained by adding the ``section at
  $\infty.$'' For generic complex structures on $L$ and $\Sigma,$ the set of
  small embeddable perturbations of the CR-structure on a strictly pseudoconvex
  hypersurface, $Y\subset\tL,$ such that the zero section of $L$ lies in the
  pseudoconcave component of $\tL\setminus Y,$ is closed in the $\CI$-topology.
\end{proposition}
\begin{proof} The hypersurface $Y$ bounds a strictly pseudoconcave domain,
  $X_-,$ in $\tL,$ which contains the zero section. The genericity assumption
  implies that the cohomology group $H^2_c(X_-;\Theta)$ vanishes. Hence we can
  apply Theorem~\ref{thm1} to conclude that the relative index between the
  Szeg\H o projector on $Y,$ and any small embeddable perturbation is uniformly
  bounded. Using Theorem E in~\cite{Epstein} we complete the proof of the
  Proposition.
\end{proof}

\begin{remark} This result generalizes Lempert's Theorem 1.1
  from~\cite{Lempert2}, covering strictly pseudoconvex hypersurfaces in
   $\bbC^2\subset\bbP^2,$ in that the hypersurface is not assumed to be the
   boundary of a tubular neighborhood of the zero section of $L.$ For
   boundaries of small tubular neighborhoods we have a stronger result: the set
   of all embeddable perturbations is closed in the $\CI$-topology provided
   that $\deg L>2g-2,$ see~\cite{Epstein44}. In the latter case there is no
   smallness hypothesis.
\end{remark}

\end{document}